\documentclass[10pt,twocolumn]{IEEEtran}

\IEEEoverridecommandlockouts                              
\def\1{{\bf 1}}
\def\R{\mathbb{R}}

\usepackage{color}
\usepackage{algorithmic}
\usepackage{ amssymb }
\usepackage{epsfig}

\usepackage {amssymb}
\usepackage {amsmath}
\usepackage {amsthm}
\usepackage{url}

\newtheorem{thm}{Theorem}
\newtheorem{lem}[thm]{Lemma}

\newtheorem{prop}[thm]{Proposition}
\usepackage{mathrsfs} 
\usepackage{nicefrac}
\usepackage{fancyvrb} 
\usepackage{subfigure}

\newcommand {\fn}{\begin{footnotesize}}
\newcommand {\en}{\end{footnotesize}}

\title{\LARGE \bf
Nonuniform Line Coverage from  Noisy Scalar Measurements
}

\author{Peter Davison, Naomi Ehrich Leonard, Alex Olshevsky, Michael Schwemmer\thanks{%
Supported in part by ONR grant N00014-09-1-1074.  P. Davison is with the Dept.\ of Aeronautics and Astronautics, MIT, {\tt\small pdavison@mit.edu}.   N.E. Leonard is with the Dept.\ of Mechanical and Aerospace Engineering, Princeton University, {\tt\small naomi@princeton.edu}.  A.
Olshevsky is with the Department of Industrial and Enterprise Systems Engineering, University of Illinois at Urbana-Champaign, {\tt\small aolshev2@illinois.edu}. M. Schwemmer is with the Mathematical Biosciences Institute, Ohio State University, {\tt\small schwemmer.2@mbi.osu.edu}, and is supported in part by the MBI and the NSF under grant DMS 0931642.  }
}

\begin{document}

\maketitle
\thispagestyle{empty}
\pagestyle{empty}

\begin{abstract} We study the problem of distributed coverage control in a network of mobile agents arranged on a line. The goal is to design distributed dynamics for the agents to achieve optimal coverage positions with respect to a scalar density field that measures the relative importance of each point on the line. Unlike previous work,
which implicitly assumed the agents know this density field, we only assume that each agent can access noisy samples of the field
at points close to its current location.  We provide a simple randomized protocol 
wherein every agent samples the scalar field at three nearby points at each step and which guarantees convergence to the optimal positions. We further analyze the convergence time of this protocol and show that, under suitable assumptions, the squared distance to the 
optimal coverage configuration decays as
$O(1/t)$ with the number of iterations $t$, where the constant scales polynomially with the number of agents $n$.  We illustrate these
results with simulations.
\end{abstract}


\section{Introduction}
As technological advances have improved the capabilities, reliability, and cost of robotic sensing platforms, their potential for deployment in autonomous, cooperative networks has gained significant attention. The emergent capabilities of  mobile sensor networks promise to revolutionize  complex tasks  such as surveillance, exploration and environmental monitoring. However, development of   high-level capabilities requires solutions to  lower-level problems such as formation control and coverage control, and these solutions should be distributed, adaptive to changing environments, and robust to uncertainty and changes in network topology.

The present work focuses on coverage control, where the goal is to optimally locate the nodes, or agents, to maximize the so-called coverage metric, which 
measures the largest distance from a point in a domain of interest to the closest node. The coverage problem may be thought of in terms
of interception time: the optimal coverage locations in a domain minimize the largest ``response time'' from the node locations to a point in the domain.

The coverage problem often  involves distances which differ from the ordinary Euclidean 
metric by a weight factor which puts a heavier weight on some regions relative to others; in this case, the problem is referred to as the {\em nonuniform} coverage control problem, whereas if the distances are Euclidean then  
it is standard to refer to the problem as the {\em uniform} coverage problem.  The weight factor of each point is usually referred to as the density field. Intuitively, one may think of minimizing interception time in a terrain of varying roughness, where different regions of the terrain take longer to traverse compared to others. In a nonuniform field, agents should be closer together in regions of higher density and more spread out in regions of lower density.

We will consider a particular case of the nonuniform coverage problem when the agents are
arranged on a line. Our interest in the line coverage problem  is motivated by two distinct considerations. 

First, the line coverage problem is the natural model for the ``border patrol'' problem, wherein
we must position $n$ nodes on a curve in $\R^2$ such that the line  distance  from any point 
on the line to the closest node is minimized. After a reparametrization of the curve, this is exactly the line coverage
problem. In this context, the curve usually represents a physical border; the $n$ nodes will usually be autonomous vehicles; the distance metric being optimized is  the
largest interception time from any point on the border to the closest vehicle; and the density of a location represents difficulty of travel, determined by the roughness of the terrain in that location. 

The terrain roughness can be measured by each vehicle at its
current location using a laser stripe generator (see \cite{zhang}). Correspondingly, we seek an adaptive protocol by means of 
which $n$ vehicles can explore the border and optimally position themselves without knowing the terrain roughness in advance.  Since the optimal position will depend on the terrain roughness at all points on the border, any protocol for this problem will need to ensure that the vehicles do not neglect any part of the border in their sampling. Consequently, we will study algorithms wherein each vehicle repeatedly samples the terrain density in the region of the border closer to it than to other agents (thus ensuring no part of the border is neglected) and, based on these samples and the positions of its neighbors, moves  to a new location. 

A related application is  adaptive sampling applications in two and three dimensions. There is often one dimension where nonuniformity dominates and a protocol  is needed for coverage in the dominant direction; e.g.,  in the ocean autonomous vehicles measuring temperature would use the protocol to optimally position themselves along the thermocline in the vertical water column.  


Secondly, the general  coverage control 
problem  has been the subject of much recent interest within the control community; we refer the reader to the recent papers \cite{Cortes2004, Cortes2005, Poduri2004, Li2005, Martinez2007, Gao2008, Carlos2008, Susca2008,  Laventall2009, Leonard2009, Cortes2010,  cortes-line, LOpaper,Obermeyer2011, Pavone2011, Bullo2012, Sydney2012, peter-allerton} and the references therein. The nonuniform line coverage problem is the one-dimensional version of the general coverage control problem, and consequently, it provides a simplified setting  to make advances in addressing outstanding questions in coverage control.  

Due to the large amount of related literature, we will only attempt to survey the works most directly relevant to our problem and approach.  In \cite{Cortes2004,Cortes2005} a distributed uniform coverage control law is developed which makes use of Voronoi partitions and gradient descent laws.   The nonuniform case is treated in part by using density-dependent gradient descent laws with the Voronoi partitions computed for the uniform case.   Coverage with communication constraints is treated in \cite{Poduri2004,Li2005}.   
The nonuniform coverage control problem is addressed in \cite{Leonard2009} where a density-dependent distance metric is defined that stretches and shrinks subregions of high and low density. A cartogram transformation (which needs to be known be the agents) is then used to compute the Voronoi partitions and convergence to optimal nonuniform coverage is proved in the case of a static or slowly time-varying density field.  
The necessity of knowing the cartogram transformation is relaxed in the case of nonuniform coverage control on the line in \cite{LOpaper} where fully distributed nonuniform coverage protocols are derived. A follow-up work \cite{peter-allerton} considers the line coverage problem where only samples of $\rho$ are available and explores the performance guarantees associated with some simple strategies. Some recent work \cite{cortes-line, julien-line} considers strategies for  coverage on the line when some proportion of the sensors are expected to randomly fail.

In this work our focus is on deriving rigorously correct protocols for optimal line coverage under the assumption that every node has access to noisy measurements of the 
field $\rho$. This is a considerably weaker assumption compared to the previous literature (e.g., \cite{Cortes2004, Cortes2005, Leonard2009, LOpaper}) where convergence to optimal coverage was established with update rules that involved exact computation of integrals of $\rho$, thus implicitly assuming
that the density is known to all the nodes.  Our main result is a randomized protocol that drives all the agents to the optimal configuration
from only three noisy samples of $\rho$ at each step. Moreover, we derive upper bounds on the convergence time of our protocol that scale polynomially with the number
of nodes $n$. We perform  simulations that show the convergence times of our protocol are quite reasonable, scaling considerably 
faster than our worst-case bounds. 

The paper is structured as follows. In Section \ref{sec:background} we provide the necessary background to state the problem formally and briefly
summarize our main results. Section \ref{sec:main} then contains our protocol and a proof of its convergence and convergence rate.  We illustrate these results with simulations  in Section \ref{sec:simulations}.

\section{Background and summary of our results}\label{sec:background}
We will consider a network of $n$ mobile agents whose locations are assumed, for simplicity, to lie within the interval $[0,1]$. We will denote the positions of these agents, or nodes, at time step $t$ by $x_1^{(t)},x_2^{(t)},\hdots,x_n^{(t)}$. We adopt the convention that the labeling of agents from $1$ to $n$ matches their initial order at time step $0$ along the line from left to right, i.e., $x_1^{(0)} \leq x_2^{(0)} \leq \cdots \leq x_n^{(0)}$.  We will use $x_0^{(t)}$ and $x_{n+1}^{(t)}$ to denote $0$ and $1$ for all $t$, respectively; this will simplify  notation.

The information density field, $\rho: \R \rightarrow (0,\infty)$ is assumed to be a positive differentiable function which is  bounded above and bounded away from zero from below.
Following \cite{Leonard2009}, we define the distance between two points $a,b$ as 
$$
d_\rho(a,b)=\int_{\min(a,b)}^{\max(a,b)}\rho(z) ~dz.
$$  Intuitively  this distance function stretches regions of high $\rho$ relative to regions of lower $\rho$.
Using this notion of distance, the coverage metric $\Phi$ is then the largest distance from any point in the domain $[0,1]$ to the agent that is nearest to it:
\vspace{-0.05truein}
$$
\Phi(x_1,\hdots,x_n,\rho)=\max_{y\in[0,1]}\bigg[\min_{i=1,\hdots,n}d_\rho(y,x_i)\bigg].
$$ The optimum coverage $\Phi^*$ is the minimum of $\Phi$ over all possible agent configurations in the interval $[0,1]$.  It is not 
hard to see (and was proven in \cite{LOpaper}) that for any $\rho$ satisfying the above assumptions, $\Phi$ is minimized at a
unique vector $x = x^*$ among vectors with nondecreasing entries, i.e., those satisfying $x_1 \leq x_2 \leq \cdots \leq x_n$. 

The nonuniform line coverage problem is to design a protocol that drives all the agents to a configuration achieving coverage $\Phi^*$
$\Phi^*$. Furthermore, this protocol must be distributed, meaning that each agent is limited to repeatedly updating its position based only
on the positions of its closest neighbors on the left and right as well as measurements of the density $\rho$ taken near its location. The
special case when $\rho$ is identically equal to $1$ is referred to as the uniform line coverage problem. 

It is easy to see that the optimal coverage configuration is invariant under any scaling of the density field $\rho$. We will therefore assume henceforth,
without loss of generality,  that the range of $\rho$ is $[1,\rho_{\rm max}]$ for some finite positive $\rho_{\rm max} \geq 1$.

A variety of protocols for the line coverage problems are available (see \cite{Cortes2004,Cortes2005, Leonard2009} and the 
recent paper \cite{LOpaper} focused on the line coverage problem). However,
all of these works implicitly assume that the agents know $\rho$ exactly because they include integrals of $\rho$ in the update equations used by each agent. By contrast, in this
paper we instead assume that each agent only has access to {\em noisy} samples of $\rho$. Specifically, we will
assume that every agent can take samples of the form $\widehat \rho(z) = \rho(z) + w$, where $w$ is noise and $z$ is a point between the measuring agent's  left and right neighbors.   We remark that this assumption may involve some physical travel on the part of each node at every step; for example, if it
chooses to sample $\rho$ at a point outside its physical sensing radius (but still between its left and right neighbors), it will need to move closer to that  point.  The noises $w$ are assumed to be independent, have zero mean and bounded support, and an upper bound on this support, which we will 
denote by $M$, is known to all the nodes. 
Furthermore, we will assume that $\widehat{\rho}(z)$ is nonnegative with probability $1$; this occurs, for example, if the noise support is not too large. Intuitively, since the density $\rho(z)$ represents difficulty of travel (and is therefore nonnegative) we require noisy estimation of it to result in nonnegative samples. 

Our main contribution in this paper is a protocol for the  nonuniform line coverage problem wherein each node uses only three samples of
the density at each step. 
Moreover, we obtain precise bounds on the convergence 
speed of our protocol under the additional assumption that every node knows a rough estimate of the total number of nodes in the
network. IWe will show that the per-node expected square distance from optimal coverage configuration decays as $O(n^5/t)$, where the constant depends on the  quantity $\rho_{\rm max}$ as well as on the support of the noise distribution $M$.

\section{Nonuniform line coverage from noisy scalar measurements\label{sec:main}}

We next describe the line coverage protocol we present and analyze here. We begin with an informal sketch. We will show that optimal coverage configurations can be characterized as the minima of a certain Lyapunov function, and the first-order conditions for optimality prescribe that the $\rho$-weighted distances between each agent and each of its two closest neighbors be in a certain proportion.  This naturally suggests an algorithm wherein each agent repeatedly moves to put theses distances in the right proportion. However, without exact knowledge of $\rho$, the nodes estimate the current $\rho$-weighted distances through random sampling. In particular, each node will sample $\rho$ at three  points, one between itself and its right neighbor, one
between itself and its left-neighbor, and one at its current location. These samples of $\rho$ allow it to compute an unbiased estimate of where its next location should be 
in order to put the two distances into the desired proportion\footnote{In line with previous protocols for coverage control which relied on computation of Voronoi partitions (e.g., \cite{Cortes2004, Cortes2005, Leonard2009}), this process may be thought of in terms of each node computing its Voronoi cell (interval) from a few noisy samples at each stage.}.

However, because all the estimates obtained by the nodes in this way are noisy, we will introduce a stepsize; thus nodes will move only part-way towards their
new positions, and the size of the move at each step will (slowly) decay to zero with time. Intuitively,  the position of each node is influenced by all previous samples of $\rho$ collected by it and neighboring nodes, and consequently as time goes on and the node's position reflects more
and more past samples, the node will need to move less in response to each new sample. 


\subsection{A formal statement of the protocol}

At time step $t$, node $k$ will sample the density $\rho$ at a random point $r_k^{(t)}$ uniformly in the interval $[x_k^{(t)}, x_{k+1}^{(t)}]$, i.e.,
between itself and its right neighbor; and at a random point 
$l_k^{(t)}$ uniformly in the interval $[x_{k-1}^{(t)}, x_k^{(t)}]$, i.e, between itself and its left neighbor (recall our convention that $x_0^{(t)} = 0, x_{n+1}^{(t)} = 1$). Finally, node $k$ samples the density at its own location $x_k^{(t)}$. After obtaining these samples node $k$ proceeds to set\footnote{Recall our notation: $\widehat \rho(z)$ is  $\rho(z)$ plus a zero-mean random variable with support in $[-M,M]$, and all these random variables are jointly independent.}
\begin{eqnarray*} R_k^{(t)} & = & \widehat \rho \left( r_k^{(t)} \right)   \left( x_{k+1}^{(t)} - x_k^{(t)} \right) \\ 
L_k^{(t)} & = & \widehat {\rho} \left( l_k^{(t)} \right) \left( x_{k}^{(t)} - x_{k-1}^{(t)} \right)
\end{eqnarray*}and then to update
\begin{eqnarray}  x_1^{(t+1)} & = & x_1^{(t)} - \frac{\alpha(t) \widehat \rho \left (x_1^{(t)} \right)}{8 \left( \rho_{\rm max} + M \right)^2} \left( 2L_1^{(t)} - R_1^{(t)} \right) \nonumber \\
x_k^{(t+1)}  & =&    x_k^{(t)} - \frac{\alpha(t) \widehat \rho \left( x_k^{(t)} \right) }{8 \left( \rho_{\rm max} + M \right)^2} \left( L_k^{(t)} - R_k^{(t)} \right) \\ && \mbox{ when } k = 2, \ldots, n-1 \nonumber \\ 
x_n^{(t+1)} & = & x_n^{(t)} - \frac{ \alpha(t) \widehat  \rho \left( x_n^{(t)} \right)}{8   \left( \rho_{\rm max} + M \right)^2} \left( L_n^{(t)} - 2 R_n^{(t)} \right) \label{mainupdate} \end{eqnarray} where 
$\alpha(t)$ is a  stepsize which satisfies \[ 0 \leq \alpha(t) \leq 1 \mbox{ for all } t, ~~~ \sum_{t=0}^\infty \alpha(t) = +\infty, ~~~ \sum_{t=0}^{\infty} \alpha^2(t) < \infty. \] We remark that choosing the stepsize $\alpha(t) = 1/t^{p}$ for any exponent $p \in (1/2,1]$ satisfies all three of 
these conditions. Finally, we will refer to this protocol as the {\em randomized scalar coverage protocol}.

\subsection{Main result}

The main result of this  paper is the following theorem.

\begin{thm} \label{mainthm} The positions $x^{(t)} = (x_1^{(t)}, \ldots, x_n^{(t)})^T$ 
converge to the unique minimizer configuration $x^* = (x_1^*, \ldots, x_n^*)^T$ of the coverage metric $\Phi$ with probability one. Moreover, if every node knows an
upper bound $U$ on the total number of nodes $n$ and chooses the stepsize  $\alpha(t)=\frac{8 U^2 (\rho_{\rm max}+M)^2}{ 8 U^2 (\rho_{\rm max}+M)^2+t}$ then 
we will have the expected error bound 
\begin{small} \begin{eqnarray*} E \left[ \frac{1}{n} ||x^{(t)} - x^*||_2^2 \right] & \leq & \frac{16 n U^4 ( \rho_{\rm max}+M)^4  (4 \rho_{\rm max}^2 + 2 ||\rho'||_{\infty} \rho_{\rm max})}{8 U^2 (\rho_{\rm max}+M)^2+ t}, \end{eqnarray*} \end{small} where $||\rho'||_{\infty} = \sup_{z \in [0,1]} |\rho'(z)|$.

\end{thm}

Under the assumption that every node approximately knows the total number of nodes, e.g., if we have $n \leq U \leq 2n$, an 
implication of this theorem is that
we need to wait $O(n^5/\epsilon)$ iterations until the average square error $E[\frac{1}{n} ||x^{(t)} - x^*||_2^2]$ is below $\epsilon$ in expectation, where the constant within the $O$-notation
depends on the  density $\rho$ and the noise support $M$. Thus the main result of this paper is that this decay is linear in the number of iterations $t$ and the 
constant in front of this decay scales polynomially with the number of nodes $n$. 

\subsection{Proofs}

We will shortly turn to the proof of Theorem \ref{mainthm}. Before doing so, however, we need to demonstrate something more basic: that the 
protocol preserves the ordering of the agents i.e., that $x_1^{(t)} \leq x_2^{(t)} \leq \cdots \leq x_n^{(t)}$ for all $t$. If this were not the case, our
protocol would not be truly {\em distributed}: since the next position of node $k$ is affected by the positions of nodes $k-1$ and $k+1$ at each step, it is crucial that 
these three nodes continue to be each other's closest neighbors.

\begin{prop} \label{order} Under the randomized scalar coverage protocol, we have that  with probability $1$
 $$0 = x_0^{(t)} \leq x_1^{(t)} \leq x_2^{(t)} \leq \cdots \leq x_n^{(t)} \leq x_{n+1}^{(t)} = 1$$ for all integers $t$. 
\end{prop} 
 
\begin{proof}[Proof of Proposition \ref{order}] By assumption the statement is true at time $t=0$, and we prove
it by induction. Suppose that the statement holds at time $t$ and consider node $i$. As a consequence of the update rule Eq. (\ref{mainupdate}) and the
fact that $R_i^{(t)}, L_i^{(t)} \geq 0$, we have 
\begin{eqnarray*}x_i^{(t+1)} - x_i^{(t)}  & \leq  &  \frac{\alpha(t)}{8 ( \rho_{\rm max}+M)} 2 R_i^{(t)} \frac{\widehat \rho \left( x_i^{(t)} \right) } { \rho_{\rm max} + M}    \\
& \leq & \frac{R_i(t)}{4 ( \rho_{\rm max}+M)} \\  
& \leq & \frac{ x_{i+1}^{(t)} - x_i^{(t)}}{4}.  \end{eqnarray*}   A similar argument establishes that 
\[  x_{i}^{(t+1)} - x_{i}^{(t)}  \geq \frac{1}{4}  \left( x_{i-1}^{(t)} - x_{i}^{(t)} \right)  \]   and these two inequalities imply the proposition. 
\end{proof}

Now that we have established that the protocol remains distributed by preserving the ordering of the nodes, we turn to the proof of Theorem \ref{mainthm}. First 
we will argue that the optimal coverage point is the minimum of a certain Lyapunov function; this is Lemma \ref{optlemma} below, which defines a function 
$Q(x_1, \ldots, x_n)$ minimized at optimal coverage. Next, we bound how much $Q(x_1, \ldots, x_n)$ decreases at each step. This appears to be difficult to do directly. 
However, relying on the key idea that $Q$ becomes convex after a position-dependent stretching of the coordinate space, we will
prove a number of inequalities which will allow us to argue that $Q$ decreases at every step by a certain fraction of the full distance to the optimal value. 

We begin now with a series of lemmas executing this plan which will culminate in the proof of Theorem \ref{mainthm}. Our first lemma introduces the 
Lyapunov function $Q$:

\begin{lem} \label{optlemma} Define  \fn 
\begin{eqnarray*} Q(x_1, \ldots, x_n) & = & 2 \left( \int_0^{x_1} \rho(z) ~dz \right)^2 + \left( \int_{x_1}^{x_2} ~\rho(z) dz \right)^2 + \\ && \cdots + \left( \int_{x_{n-1}}^{x_n} \rho(z) ~ dz \right)^2 + 2 \left( \int_{x_n}^1 \rho(z) ~ dz \right)^2.  \end{eqnarray*} \en
Then $Q(x_1, \ldots, x_n)$ has a unique global minimizer $x^* = (x_1^*, x_2^*, \ldots, x_n^*)^T$; this minimizer $x^*$ is also a minimizer
of $\Phi$ and satisfies $0 \leq x_1^* \leq x_2^* \leq \cdots \leq x_n^* \leq 1$. 
\end{lem}
\begin{proof} It is easy to see that $Q(x_1, \ldots, x_n)$ must have a global minimum, since it is continuous and blows up if at least one of the variables goes
off to infinity.  Setting all partial derivatives of $Q$ to zero and rearranging, we obtain the set of equations 
\begin{eqnarray} 2 \int_0^{x_1} \rho(z) ~ dz & = & \int_{x_1}^{x_2} \rho(z) ~ dz \nonumber \\
\int_{x_{i-1}}^{x_{i}} \rho(z) ~ dz & = & \int_{x_{i}}^{x_{i+1}} \rho(z) ~ dz \mbox{ for } i = 2, \ldots, n-1 \nonumber \\
\int_{x_{n-1}}^{x_n} \rho(z) ~ dz & = & 2 \int_{x_n}^1 \rho(z) dz. \label{prop}
\end{eqnarray}  Any global minimum of $Q$ must satisfy these equations. It is in Lemma 2 in \cite{LOpaper} that these equations  have a unique solution which minimizes $\Phi$, and the proof of that lemma  establishes that the entries of this solution are monotonic and lie in $[0,1]$.  
\end{proof}


 For convenience of notation, we will define $F$ to be  the function which 
maps $[0,1]$ into $[0, \rho_{\rm max}]$ by $F(x) = \int_0^x \rho(z) ~dz$. Note that Eq. (\ref{prop}) can be conveniently restated in terms of $F$, e.g., the equality for $i=2, \ldots, n-1$ in Eq. (\ref{prop}) is simply $F(x_i) - F(x_{i-1}) = F(x_{i+1}) - F(x_i)$.

 Next, we will need a technical estimate on the smallest value assumed by a certain quadratic form on the unit sphere. 

\begin{lem} \label{eiglower}
 \[ \min_{||x||_2 = 1}  x_1^2 + (x_2 - x_1)^2 + \cdots + (x_n - x_{n-1})^2 + x_n^2 \geq \frac{1}{n^2} \]
\end{lem}

\begin{proof} Since $||x||_2=1$ we have that at least one component $x_i$ satisfies $|x_i| \geq 1/\sqrt{n}$. Without loss of generality, let us 
assume $x_i > 0$; else, we can simply replace $x$ with $-x$. We then have that \[  \frac{1}{\sqrt{n}}  \leq  x_i 
 =   (x_1 - 0) + (x_2 - x_1) + \cdots + (x_i - x_{i-1}). \]   Applying Cauchy-Schwarz yields the lemma. 
\end{proof} 

Next we prove a technical lemma which lower bounds the size of the gradient of $Q$ at each step as a fraction of the distance to the optimal
value. The proof uses the previous Lemma \ref{eiglower} and proceeds by relying on a coordinate-depending stretching of the space making
$Q$ convex. 

\begin{lem} \[ \frac{ || Q'(x)||_2^2 }{Q(x) - Q(x^*)} \geq \frac{4}{n^2} \] \label{Qbound} \end{lem} \begin{proof} Define \fn \[ G(y_1, \ldots, y_n) = 2 y_1^2 + (y_2-y_1)^2 + \cdots + (y_n - y_{n-1})^2 + 2 \left(\int_0^1 \rho(z) ~ dz - y_n \right)^2. \] \en
Observe that \begin{equation} \label{qgrelation} Q(x_1, \ldots, x_n) = G(F(x_1), F(x_2), \ldots, F(x_n)). \end{equation} As a consequence of this and Lemma \ref{optlemma}, we can conclude that
$G(y)$ has a unique minimum $y^*$ satisfying $y_i^* = \int_0^{x_i^*} \rho(z) ~ dz$. The Hessian of $G(y_1, \ldots, y_n)$ is easily computed; it is \[ G''(y) = \begin{pmatrix} 6 & -2&  0 & 0 & \cdots & 0 \\ 
-2 & 4 & -2 & 0 & \cdots & 0 \\
0 & -2 & 4 & -2 & \cdots & 0 \\
\vdots & \vdots & \ddots & \ddots & \ddots & \vdots \\
0 & 0 & 0 & -2 & 4 & -2  \\
0 & 0 & 0 & 0 & -2 & 6
\end{pmatrix} \] We will use the standard notation of ${\bf e}_i$ to mean the column vector with a $1$ in the $i$'th component and zero elsewhere; moreover,
we will use ${\bf e}_{i,j}$  to denote the vector with a $1$ in the $i$'th component, a $-1$ in the $j$'th component, and zeros elsewhere. Then
it is easy to verify that \vskip -0.6pc 
\[ G'' = 4 {\bf e}_1 {\bf e}_1^T + 2 \sum_{i=1}^{n-1} {\bf e}_{i,i+1} {\bf e}_{i,i+1}^T + 4 {\bf e}_n {\bf e}_{n}^T \] so that its smallest 
eigenvalue satisfies \begin{footnotesize}
\begin{equation} \label{lambdabound} \lambda_{\rm min} = \min_{||y||_2=1} y^T G''y =\min_{||y||_2=1}  4 y_1^2 + 2 \sum_{i=1}^{n-1} (y_i - y_{i+1})^2 + 4 y_n^2 \geq \frac{2}{n^2} \end{equation} \end{footnotesize} where the last inequality follows from 
Lemma \ref{eiglower}.  Thus $G$ is a strongly convex function, and a standard bound on the norm of its gradient (see Lemma 3 in Chapter 1.4 of \cite{Polyak})) is \[   \frac{||G'(y)||_2^2}{G(y) - G(y^*)} \geq 2 \lambda_{\rm min}. \] Further,  for any  $x \in \R^n$ choosing $y \in \R^n$ defined by $y_i = F(x_i)$, we have by Eq. (\ref{qgrelation}) and our assumption that
$\rho \geq 1$ everywhere, that
\[\frac{ || Q'(x)||_2^2 }{Q(x) - Q(x^*)}  \geq \frac{||G'(y)||_2^2}{G(y) - G(y^*)}.  \] Putting together the last two inequalities
and plugging in the bound on $\lambda_{\rm min}$ from Eq. (\ref{lambdabound}) we obtain the current lemma. 
\end{proof}

Having now established the lower bound on the norm of the gradient of $Q(x)$, we now proceed to the proof of Theorem \ref{mainthm}. We
proceed by arguing that our protocol is a randomized version of gradient descent on the function $Q(x)$, which can be shown to converge despite the
nonconvexity of $Q$. 

\smallskip

\begin{proof}[Proof of Theorem \ref{mainthm}]  We begin by rewriting the update equation in more convenient form. Specifically, comparing  the randomized control law, namely Eq. (\ref{mainupdate}), with the definition of the function $Q(x_1, \ldots, x_n)$, we observe that we can write \begin{small}
\[ x^{(t+1)} = x^{(t)} - \frac{\alpha^{(t)}}{16 (\rho_{\rm max}+M)^2} \overline{g}^{(t)}, \] \end{small} where $E[ \overline{g}(t) ]  =  Q' \left (x^{(t)}  \right)$ and \begin{small}
 \begin{eqnarray*}  \left| \left| \overline{g}^{(t)} \right| \right|_2^2  & = &   \left( \widehat \rho \left( x_1^{(t)} \right) ( 4 L_1^{(t)} - 2R_1^{(t)}) \right)^2  \\ && + \sum_{i=2}^{n-1}  \left( \widehat \rho \left( x_i^{(t)} \right) ( 2 L_i^{(t)} - 2 R_i^{(t)} ) \right) ^2   \\ && +   \left( \widehat \rho \left( x_n^{(t)} \right) ( 2L_n^{(t)} - 4R_n^{(t)}) \right)^2  \\
& \leq & (\rho_{\rm max}+M)^2 \sum_{i=1}^n \left( \max \left( 4 L_i^{(t)}, 4 R_i^{(t)} \right) \right)^2 \\
& \leq & 64 (\rho_{\rm max}+M)^4,  
\end{eqnarray*}   \end{small}
and all the above inequalities hold with probability $1$. Next, we observe that 
$Q$ is a twice differentiable function, so that we may expand it in a Taylor series. Since \begin{footnotesize}
\begin{eqnarray*}  \frac{ \partial^2 Q} {\partial x_1^2}(x)  & = &  6 \rho^2(x_1) + 2 \rho'(x_1) \left( 2 \int_{0}^{x_1} \rho(z) ~ dz  - \int_{x_1}^{x_{2}} \rho(z) ~ dz \right) \\  
 \frac{ \partial^2 Q} {\partial x_n^2}(x)  & = &  6 \rho^2(x_n) + 2 \rho'(x_n) \left(  \int_{x_{n-1}}^{x_n} \rho(z) ~ dz  - 2 \int_{x_n}^{1} \rho(z) ~ dz \right)  \\  
\frac{ \partial^2 Q} {\partial x_i^2}(x)  & = &  4 \rho^2(x_i) + 2 \rho'(x_i) \left( \int_{x_{i-1}}^{x_i} \rho(z) ~ dz  - \int_{x_i}^{x_{i+1}} \rho(z) ~ dz \right) \\ && \mbox{ when } i = 2, \ldots, {n-1} \\ 
\frac{ \partial^2 Q} {\partial x_i \partial x_j}(x)  & = &  
-2 \rho(x_i) \rho(x_j)  \mbox{  when } j \in \{i-1,i+1\} \\
\frac{\partial^2 Q}{\partial x_i \partial x_j}(x) & = & 0 \mbox{   when } j \notin \{i-1,i,i+1\} \end{eqnarray*} \end{footnotesize} by Gershgorin circles it follows that as long as $0 \leq x_1^{(t)} \leq \cdots \leq x_n^{(t)} \leq 1$, the largest eigenvalue of $Q''(x)$ is never more than $8 \rho_{\rm max}^2 + 4 || \rho'||_{\infty} \rho_{\rm max} $ in magnitude.  Thus as long as both $x^{(t)}$ and $x^{(t+1)}$ have entries between $0$ and $1$ and monotonically nondecreasing (note that by Proposition \ref{order} this holds at every time $t$),  we may bound $Q(x^{(t+1)})$ via the Taylor expansion with Lagrange remainder form as  \begin{eqnarray*} Q(x^{(t+1)}) &  \leq & Q(x^{(t)}) + \nabla Q(x^{(t)})^T (x^{(t+1)} - x^{(t)} ) \\ && + ( 4 \rho_{\rm max}^2 + 2 ||\rho'||_{\infty} \rho_{\rm max}) ||x^{(t+1)} - x^{(t)}||_2^2 \end{eqnarray*}  or \begin{eqnarray*}
Q(x^{(t+1)}) & \leq & Q(x^{(t)}) - \frac{\alpha(t)}{16 (\rho_{\rm max}+M)^2} \sum_{i=1}^n [Q'(x^{(t)})]_i \overline{g}_i^{(t)} \\ && + \frac{ (4 \rho_{\rm max}^2 + 2 ||\rho'||_{\infty} \rho_{\rm max}) \alpha^2(t)}{16^2( \rho_{\rm max} + M)^4} ||\overline{g}^{(t)}||_2^2. \end{eqnarray*}
Taking expectations  and
using Lemma \ref{Qbound}, \fn
\begin{eqnarray*} E \left[ Q(x^{(t+1)}) - Q(x^*) ~|~ x^{(t)} \right] & \leq & \left( 1- \frac{\alpha(t)}{4 n^2 (\rho_{\rm max}+M)^2}  \right) \left( Q(x^{(t)}) - Q(x^*) \right)  \\ && + \frac{ (4 \rho_{\rm max}^2 + 2 ||\rho'||_{\infty} \rho_{\rm max}) \alpha^2(t)}{4}. \end{eqnarray*} \en

 \begin{figure*}
\begin{center}$
\begin{array}{c}
\includegraphics[width=0.43\textwidth]{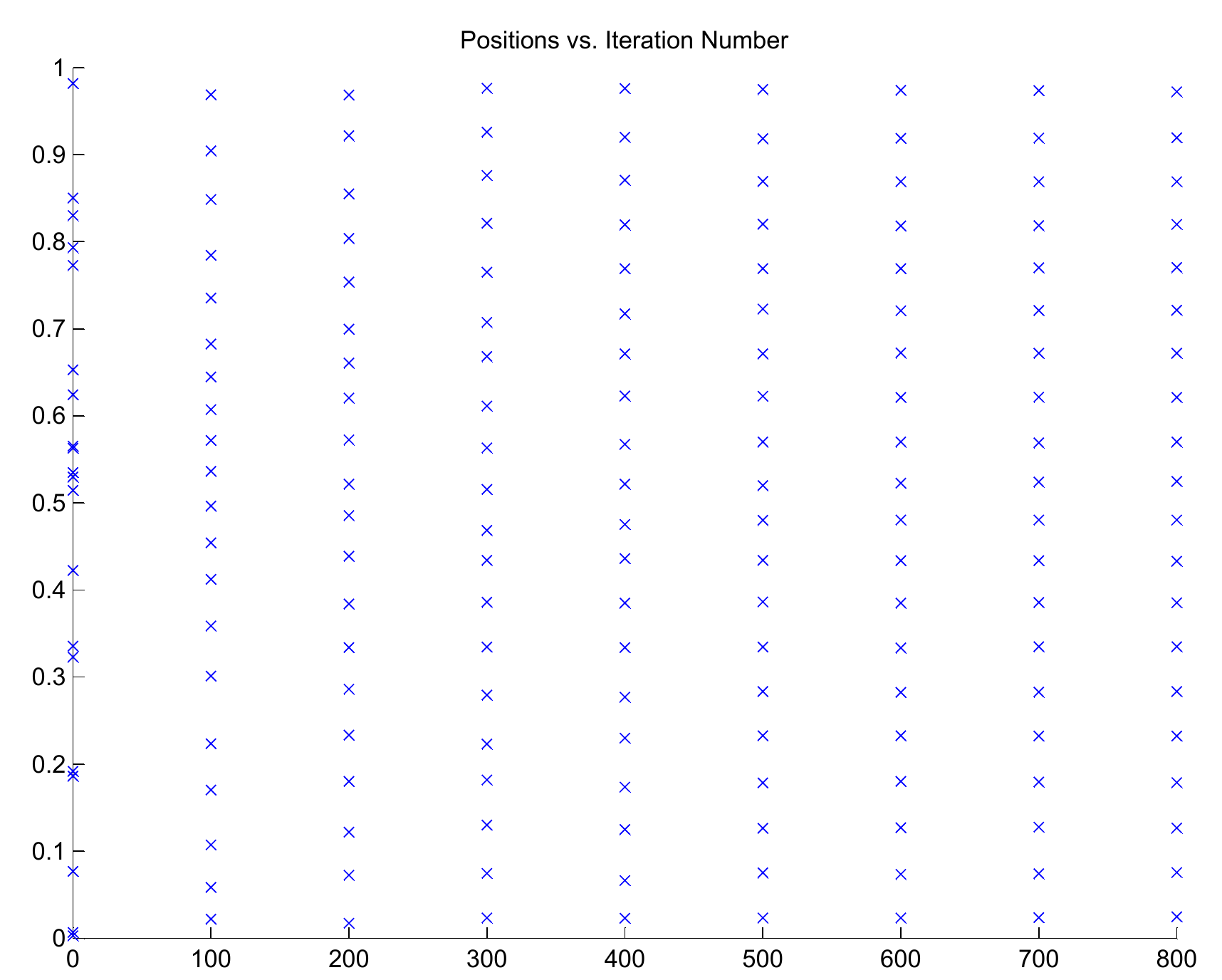} 
\includegraphics[width=0.43\textwidth]{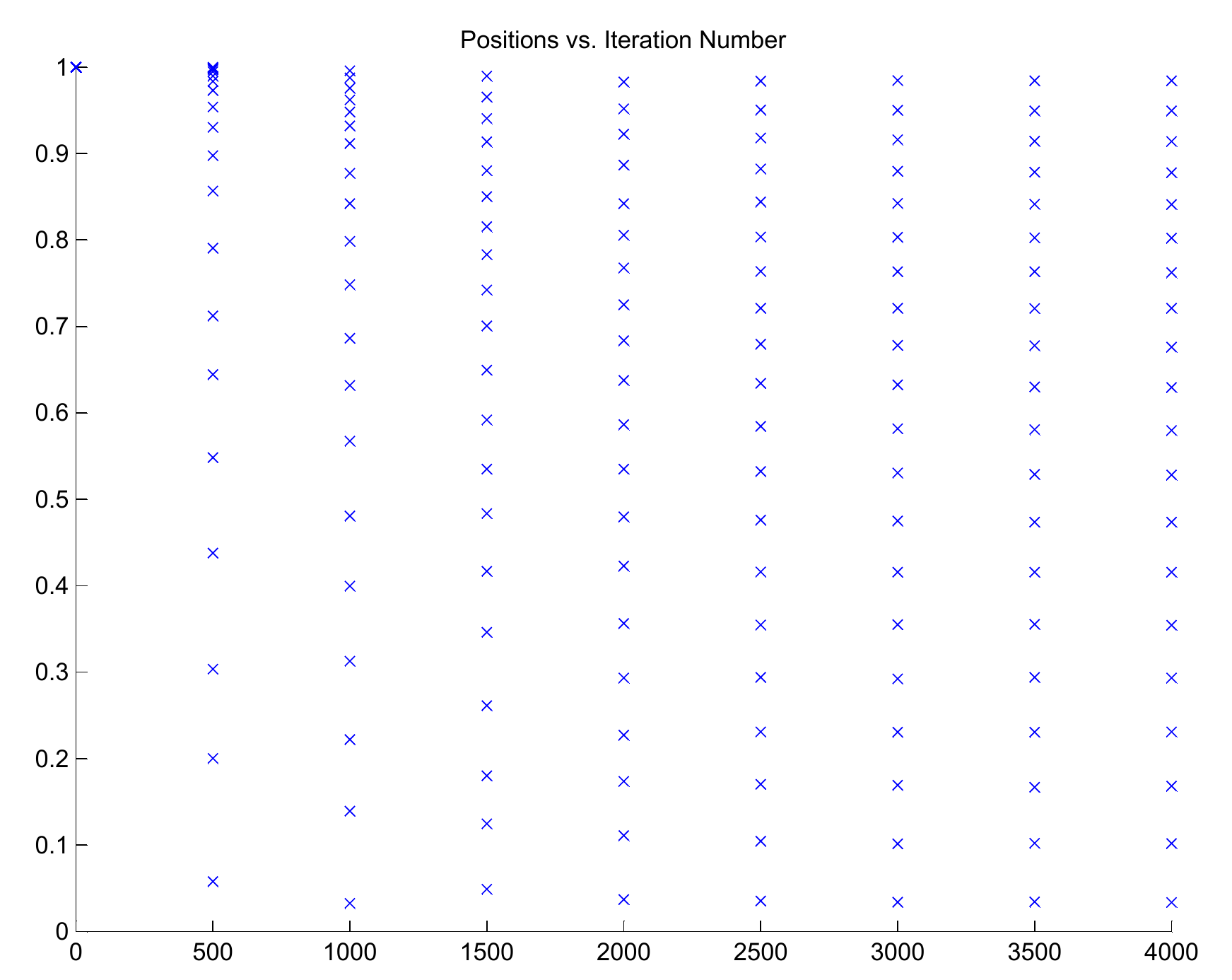}
\end{array}$
\end{center}
\caption{Both left and right figure show the positions of the nodes as a function of iteration number for a network of $20$ nodes. On the left the starting positions are uniformly random in $[0,1]$, while
on the right all nodes begin with $x_i(0)=1$.} \label{f1}
\end{figure*}
By Lemma 10  of Chapter 1.4 of \cite{Polyak} it follows that $Q(x^{(t)})-Q(x^*)$ approaches zero with
probability $1$, which by the uniqueness of the minimizer $x^*$ implies that $x^{(t)} \rightarrow x^*$ with probability $1$. Now under the additional assumption that every node knows an upper bound $U$ and 
chooses the stepsize of $\alpha(t) = \frac{8 U^2 ( \rho_{\rm max} + M)^2}{8 U^2 (\rho_{\rm max} + M)^2+t}$ we have that \begin{small}
\[  E \left[ Q(x^{(t+1)}) - Q(x^*) \right]   \leq \left( 1 - \frac{2}{8U^2 (\rho_{\rm max}+M)^2+t} \right) E \left[ Q(x^{(t)}) - Q(x^*) \right] \] \[ ~~~~~~~~~~~~~~~~~~~~~~~~~~~~~~~~~~~~~~~~~~~~~~~~~~+ \frac{16 U^4 ( \rho_{\rm max}+M)^4  (4 \rho_{\rm max}^2 + 2 ||\rho'||_{\infty} \rho_{\rm max})}{ (8 U^2 (\rho_{\rm max}+M)^2+t)^2}.  \]  \end{small} We now claim that for all $t \geq 0$, we have that \fn \begin{eqnarray*} E \left[ Q(x^{(t)}) - Q(x^*) \right] & \leq &  \frac{16 U^4 ( \rho_{\rm max}+M)^4  (4 \rho_{\rm max}^2      +2 ||\rho'||_{\infty} \rho_{\rm max})}{8 U^2 (\rho_{\rm max}+M)^2+ t} .  \end{eqnarray*} \en
We prove this claim by induction. Indeed, at $t=0$, since the initial positions are in the interval $[0,1]$, it is immediate that $Q(x(0)) \leq 2 \rho_{\rm max}^2$, which proves the statement at $t=0$.
Now suppose that the inequality holds at time $t$. For simplicity of notation, let us adopt the shorthands $C =16 U^4 ( \rho_{\rm max}+M)^4  (4 \rho_{\rm max}^2 + 2 ||\rho'||_{\infty} \rho_{\rm max})$ and $U' = 8 U^2 (\rho_{\rm max}+M)^2$. We then have that \begin{footnotesize} 
\begin{eqnarray*} E \left[ Q(x^{(t+1)} - Q(x^*) \right] & \leq & \left( 1 - \frac{2}{U'+t} \right) E \left[ Q(x^{(t)}) - Q(x^*) \right] \\ && + \frac{C}{ (U'+t)^2} \\ 
& \leq & \left( 1 - \frac{2}{U'+t} \right) \frac{C}{ (U' + t)}+ \frac{C}{ (U'+t)^2} \\
& \leq & C \left( \frac{1}{U'+t} -   \frac{1}{(U'+t)^2} \right) \\
& \leq & \frac{C}{ (U'+t+1)} 
\end{eqnarray*} \end{footnotesize} which proves the claim. Finally, we show  this implies the theorem. Observe that  $G(y) - G(y^*) \geq \frac{1}{n^2} || y - y^*||_2^2$ due to the fact that $\lambda_{\rm min}(G''(y)) \geq 2/n^2$ from Eq. (\ref{lambdabound}).  Thus, 
\[ Q(x) - Q(x^*)   \geq \frac{1}{n^2} ||F(x) - F(x^*)||_2^2 \geq \frac{1}{n^2} ||x-x^*||_2^2, \] where the last step 
follows from the fact that $\rho \geq 1$.
\end{proof} 

\section{Simulations} We briefly report on a simulation intended to gauge the practical convergence time of our protocol. Figure 1 shows simulation for a system of $20$ nodes for two starting points: one chosen uniformly at random and one which  initially places all the nodes at one corner. All noises are uniform in $[-1/2,1/2]$.  In both cases, the density is uniform, which allows convergence to the optimal configuration to be ``read off' from the graph by watching the spacings equalize. We chose the stepsize $\alpha(t)$ by setting $\alpha(t)=1$ for the first half of the iterations and setting $\alpha(t)=1/\sqrt{t}$ for the latter half; this decays more slowly as compared to the stepsize we used to obtain Theorem \ref{mainthm} but appears to be advantageous for a network of $20$ nodes executing several thousand iterations. \label{sec:simulations} 

Our simulations  confirm our theoretical convergence results. Furthermore, they suggest that our error bounds are conservative; indeed, the system 
of twenty nodes reaches close to the optimal configuration after thousands of iterations, while the upper bounds of Theorem \ref{mainthm}  are
several orders of magnitude larger. An open question is to obtain improved theoretical guarantees that better match the faster speed we observe.

\bibliographystyle{plain}
\bibliography{SeniorThesisBib}

\end{document}